\documentclass[12pt]{amsart}

\usepackage{amsmath}
\usepackage{amssymb}
\usepackage{bm}
\usepackage{graphicx}
\usepackage{psfrag}
\usepackage{color}
\usepackage{hyperref}
\hypersetup{colorlinks=true, linkcolor=blue, citecolor=magenta, urlcolor=wine}
\usepackage{url}
\usepackage{algpseudocode}
\usepackage{fancyhdr}
\usepackage{mathtools}
\usepackage{enumitem}
\usepackage{tikz-cd}
\usepackage{xy}
\input xy
\xyoption{all}
\usepackage{stmaryrd}
\usepackage{calrsfs}
\usepackage{wasysym}

\newtheorem{theorem}{Theorem}[section]

\newtheorem{prop}[theorem]{Proposition}
\newtheorem{cor}[theorem]{Corollary}

\theoremstyle{definition}
\newtheorem{definition}[theorem]{Definition}
\newtheorem{remark}[theorem]{Remark}
\newtheorem{example}[theorem]{Example}
\numberwithin{equation}{section}

\newcommand{\nn}{\mathbb{N}}

\newcommand{\pp}{\mathbb{P}}
\newcommand{\qq}{\mathbb{Q}}
\newcommand{\rr}{\mathbb{R}}
\newcommand{\zz}{\mathbb{Z}}

\newcommand{\mtl}{\mathsf{L}}

\providecommand\ldb{\llbracket}
\providecommand\rdb{\rrbracket}

\newcommand{\lr}[1]{\langle #1\rangle}

\newcommand{\pspaces}{\!\!\!\!\!\!\!\!}

\definecolor{ocre}{RGB}{223,32,128}


\begin{document}

\mbox{}
\title{Factorizations in reciprocal Puiseux monoids}

\author{Cecilia Aguilera}
\address{Akademix \\ Cambridge \\ MA 02142}
\email{aguilera.cecilia@outlook.com}

\author{Marly Gotti}
\address{Research and Development \\ Biogen \\ Cambridge \\ MA 02142}
\email{marly.cormar@biogen.com}

\author{Andre F. Hamelberg}
\address{Department of Mathematics\\MIT\\Cambridge, MA 02139}
\email{afh@mit.edu}


\begin{abstract}
	 A Puiseux monoid is an additive submonoid of the real line consisting of rationals. We say that a Puiseux monoid is reciprocal if it can be generated by the reciprocals of the terms of a strictly increasing sequence of pairwise relatively primes positive integers. We say that a commutative and cancellative (additive) monoid is atomic if every non-invertible element $x$ can be written as a sum of irreducibles. The number of irreducibles in this sum is called a length of $x$. In this paper, we identify and investigate generalized classes of reciprocal Puiseux monoids that are atomic. Moreover, for the atomic monoids in the identified classes, we study the ascending chain condition on principal ideals and also the sets of lengths of their elements.
\end{abstract}

\maketitle

\section{Introduction}
\label{sec:intro}

A cancellative and commutative (additive) monoid $M$ is called atomic if every non-invertible element of $M$ can be expressed as a sum of irreducible elements, which are often called atoms. Although the notion of atomicity was first studied in the context of commutative rings (see, for instance, \cite{pC68} and \cite{aG74}), it is perhaps in \cite{fHK92} that F. Halter-Koch provides the first study of certain atomic concepts in the more abstract context of commutative monoids. Since then a substantial amount of papers have been devoted to the study of atomicity and related factorization properties on the setting of commutative (and also non-commutative) monoids. The book~\cite{GH06b} comprises a significant part of the study of non-unique factorizations in atomic monoids (and also in integral domains) that was carried out until 2006. A current survey on factorizations in commutative monoids by A. Geroldinger and Q. Zhong can be found in~\cite{GZ20}.
\smallskip

A \emph{Puiseux monoid} is an additive submonoid of the real line consisting of nonnegative rationals. Although Puiseux monoids have appeared in the literature only sporadically since the 1970s, the first systematic study of their atomicity started in~\cite{fG18} and have continued with a series of papers by several authors, including S. T. Chapman, F.~Gotti, H. Polo, and the second author (see the recent papers \cite{CGG21,GT22,hP21} and references therein). Although Puiseux monoids are concrete algebraic objects that are easy to define, their atomic structure is rather complex. Despite the effort of the several papers devoted to their study in the last few years, still there is no nontrivial characterization of the Puiseux monoids that are atomic (some of them, such as $\langle \frac{1}{2^n} \mid n \in \nn \rangle$, are not atomic). This explains why Puiseux monoids have been the focus of several papers recently. They have been also studied in connection with factorization of matrices~\cite{BG20} and music theory~\cite{mBA20}. Several interesting examples of Puiseux monoids have been given in~\cite{BCG21} and~\cite{GGT21}. The interested reader can find surveys on Puiseux monoids and on their generalizations to positive monoids (i.e., additive submonoids of $\rr$) in~\cite{CGG21} and \cite{CG21}, respectively.
\smallskip

In addition, Puiseux monoids are powerful tools to disprove conjectures in commutative ring theory. For instance, A. Grams in~\cite{aG74} uses a Puiseux monoid as the main ingredient to construct the first atomic domain without the ascending chain condition on principal ideals (ACCP), disproving P. Cohn's assertion that every atomic domain satisfies the ACCP \cite[Proposition~2]{pC68}. Examples of Puiseux monoids are also used by D. D. Anderson, D. F. Anderson, and M. Zafrullah in their landmark paper~\cite{AAZ90} to construct examples of integral domains satisfying certain desired factorization properties. More recently, J. Coykendall and F. Gotti have used Puiseux monoids in~\cite{CG19} to construct the first atomic monoids whose monoid algebras are not atomic, partially answering a question that R. Gilmer posed in the 1980s (see \cite{rG84}). Even more recently, F. Gotti and B. Li use Puiseux monoids to generalize the construction by Grams mentioned above and, therefore, obtain further classes of atomic domains without the ACCP \cite[Examples~3.5 and~3.6]{GL21}.
\smallskip

In this paper we continue the study of Puiseux monoids. Our main purpose here is to generalize relevant examples of Puiseux monoids that have been used in commutative ring theory, specially in the construction of needed atomic monoid domains \cite{aG74,rG84,AAZ90,CG19,fG21,GL21}. Two recurrent examples in the mentioned constructions are
\[
	\Big\langle \frac 1p \ \Big{|} \ p \in \pp \Big\rangle \quad \text{ and } \quad \Big\langle \frac{1}{2^n p_n} \ \Big{|} \ n \in \nn \Big\rangle,
\]
where $\pp$ denotes the set of primes and $(p_n)_{n \ge 1}$ is a strictly increasing sequence of odd primes. Accordingly, we say that a Puiseux monoid is reciprocal if it has a generating set of the form $\{\frac{1}{d_n} \mid n \in \nn \}$ for a strictly increasing sequence $(d_n)_{n \ge 1}$ of positive integers whose terms are pairwise relatively primes. We also consider two natural generalizations of reciprocal Puiseux monoids, which we call weak and almost reciprocal: for weak reciprocal Puiseux monoids we do not impose the relatively prime condition on the denominators and for almost reciprocal Puiseux monoids we allow the numerators of the generating sequence to be different from $1$.
\smallskip

In Section~\ref{sec:classes of reciprocal PMs}, we give several examples of both weak and almost reciprocal Puiseux monoids and we compare these new notions with the notions of bounded and strongly bounded defined in \cite{GG18} in the context of Puiseux monoids. It turns out that there is some relation between the notions we introduce here and those of boundedness introduced in~\cite{GG18}. In the same section, we also consider the atomicity of weak and almost reciprocal Puiseux monoids, and construct a new class of atomic Puiseux monoids.
\smallskip

Section~\ref{sec:ACCP} revolves around the ascending chain condition on principal ideals in the context of reciprocal Puiseux monoids. We introduce the notion of atomic decomposition, already implicit in~\cite{CGG20}, and we use this notion to find a new class of (almost reciprocal) Puiseux monoids satisfying the ACCP. We also identify a class of atomic weak reciprocal Puiseux monoids that do not satisfy the ACCP. In addition, in Section~\ref{sec:ACCP} we consider the sets of lengths of reciprocal Puiseux monoids. We determine the sets of lengths of reciprocal Puiseux monoids, showing that the elements of a reciprocal Puiseux monoid that have more than one factorization must have a factorization whose number of atoms is as large as we want. We conclude by considering reciprocal Puiseux monoids whose elements have only finitely many factorizations with any prescribed number of atoms, as this property has been recently introduced and considered by A. Geroldinger and Q. Zhong in~\cite{GZ21}.

%

\bigskip
\section{Background}
\label{sec:background}

We proceed to introduce the notation and terminology we will be using later in the paper. For related information beyond this brief background we offer in this section, the interested reader is encouraged to consult~\cite{GH06b} by A. Geroldinger and F. Halter-Koch. We let $\nn$ denote the set of positive integers, and we set $\nn_0 := \{0\} \cup \nn$. Also, let $\mathbb{P}$ denote the set of prime numbers. For a positive rational number $r = n/d$ with $n,d \in \nn$ and $\gcd(n,d) = 1$, we call $n$ the \emph{numerator} and $d$ the \emph{denominator} of $r$, and we set $\mathsf{n}(r) := n$ and $\mathsf{d}(r) := d$.
For $i,j \in \nn_0$ with  $i \le j$, we let $\llbracket i, j \rrbracket$ denote the set of integers between $i$ and $j$, that is,
\[
	\llbracket i,j \rrbracket \coloneqq \left\{ k \in \zz : i \leq k \leq j \right\}.
\]

Throughout this paper, we refer as \emph{monoid} to any cancellative and commutative semigroup with identity, and we implicitly assume that every monoid here is written additively and, therefore, its identity element is denoted by $0$. A monoid is said to be \emph{reduced} provided that its only invertible element is $0$. We tacitly assume that every monoid showing up in this paper is reduced. Let $M$ be a monoid. For $q,r \in M$, we say that $q$ \emph{divides} $r$ in $M$ and write $q \mid_M r$ if we can write $r = q + q'$ for some $q' \in M$. An element $a \in M$ is called an \emph{atom} (or \emph{irreducible}) if whenever $a = q + r$ for some $q,r \in M$, either $q = 0$ or $r = 0$. Following standard notation, we let $\mathcal{A}(M)$ denoted the set of atoms of $M$. The following definition is one of the central notions of these paper.

\begin{definition}
	A monoid $M$ is called \emph{atomic} if every element of $M$ can be written as a sum of finitely many atoms.
\end{definition}

For a subset $S$ of $M$, we let $\langle S \rangle$ denote the minimal submonoid of $M$ containing $S$, that is, the intersection of all submonoids of $M$ containing $S$. If $M = \langle S \rangle$, then $S$ is called a \emph{generating set} of~$M$. If $M$ is generated by a sequence $(q_n)_{n \ge 1}$, we simply write $M = \langle q_n \mid n \in \nn \rangle$. Since monoids in this paper are reduced, one can readily verify that for every monoid $M$ and any generating set $S$ of $M$, the inclusion $\mathcal{A}(M) \subseteq S$ holds. We will use this fact often throughout the paper.

An \emph{ideal} of $M$ is a subset $I$ of $M$ such that $I + M := \{x+q \mid x \in I \text{ and } q \in M\}$ is a subset of~$I$ (or, equivalently, $I + M = I$). An ideal $I$ is called \emph{principal} if there exists $r \in M$ such that $I = r + M := \{r+q \mid q \in M\}$. The monoid $M$ is said to satisfy the \emph{ascending chain condition on principal ideals} (\emph{ACCP}) provided that each ascending chain of principal ideals of $M$ becomes stationary. If $M$ satisfies the ACCP, then it must be atomic~\cite[Proposition 1.1.4]{GH06b}. Grams' monoid is an atomic monoid that does not satisfy the ACCP (see Example~\ref{ex:Grams' monoid} and Remark~\ref{rem:Grams monoid is not ACCP}).

Let $M$ be an atomic monoid, and let $x \in M$ be a nonzero element. If $x = a_1 + \dots + a_\ell$ for some $\ell \in \nn$ and $a_1, \dots, a_\ell \in \mathcal{A}(M)$, then we call the formal sum $z := a_1 + \dots + a_\ell$ a \emph{factorization} of $x$, and call $|z| := \ell$ the \emph{length} of such a factorization. The set of all possible factorizations of $x$ is denoted by $\mathsf{Z}_M(x)$ or simply by $\mathsf{Z}(x)$, and the set
\[
	\mathsf{L}(x) := \mathsf{L}_M(x) :=  \{|z| \, \mid z \in \mathsf{Z}(x)\}
\]
is called the \emph{set of lengths} of $x$. Also, we say that $M$ is a \emph{bounded factorization monoid} (\emph{BFM}) provided that $\mathsf{L}(x)$ is a finite set for every nonzero $x \in M$. Every BFM is known to satisfy the ACCP \cite[Corollary~1]{fHK92}. The converse does not hold in general; for instance, $\langle \frac{1}{p} \mid p \in \pp \rangle$ satisfies the ACCP but is not a BFM (Example~\ref{ex:prime reciprocal PM}). Finally,~$M$ is called a \emph{finite factorization monoid} (\emph{FFM}) provided that $\mathsf{Z}(x)$ is finite for every $x \in M$. Clearly, every FFM is a BFM. The converse does not hold in general; for instance, it follows from \cite[Proposition~4.5]{fG19a} that the monoid $\{0\} \cup \qq_{\ge 1}$ is a BFM but it is not hard to check that it is not an FFM (see \cite[Example~4.7]{AG22}). Both the bounded and the finite factorization property were introduced by Anderson, Anderson, and Zafrullah in~\cite{AAZ90} in the context of commutative ring theory, and they have been further studied as part of the sequel~\cite{AA10,AAZ92,AeA99}.

\bigskip
\section{Reciprocal Puiseux Monoids and Related Classes}
\label{sec:classes of reciprocal PMs}

For any strictly increasing sequence $(d_n)_{n \ge 1}$ of positive integers, we call the Puiseux monoid $M$ generated by the set $\big\{ \frac1{d_n} \mid n \in \nn \big\}$ the \emph{weak reciprocal} Puiseux monoid of the sequence $(d_n)_{n \ge 1}$ or, simply, a \emph{weak reciprocal} Puiseux monoid.

\begin{example} \label{ex:classic antimatter PM}
	Fix $b \in \nn$ with $b \ge 2$, and consider the Puiseux monoid defined by $M = \big\langle \frac 1{b^n} \mid n \in \nn \rangle$. It is clear that $M$ is a weak reciprocal Puiseux monoid. As $\frac{1}{b^{n+1}}$ divides $\frac{1}{b^n}$ in $M$ for every $n \in \nn$, we see that the set of atoms of $M$ is empty and, therefore, $M$ is not atomic.
\end{example}

On the other hand, there are weak reciprocal Puiseux monoids that are atomic, as the following example illustrates.

\begin{example} \label{ex:Grams' monoid}
	Let $(p_n)_{n \ge 1}$ be the strictly increasing sequence whose underlying set is $\pp \setminus \{2\}$. The Puiseux monoid $M = \langle \frac{1}{2^n p_n} \mid n \in \nn \rangle$ is a weak reciprocal Puiseux monoid. It follows from Proposition~\ref{prop:atomic PMs} below that $M$ is atomic with $\mathcal{A}(M) = \{\frac{1}{2^n p_n} \mid n \in \nn\}$. Since every generating set of $M$ must contain $\mathcal{A}(M)$, we see that $M$ is not a reciprocal Puiseux monoid. The monoid $M$ was first used by A. Grams in~\cite{aG74} to construct the first example of an atomic integral domain that does not satisfy the ACCP, correcting an assertion made by P. M. Cohn in~\cite{pC68}. Because of this, this monoid is often called that \emph{Grams' monoid}.
\end{example}

Now suppose that $(d_n)_{n \ge 1}$ is a sequence of positive integers whose terms are pairwise relatively primes. Then we call $\big\langle \frac{1}{d_n} \mid n \in \nn \big\rangle$ the \emph{reciprocal Puiseux monoid} of $(d_n)_{n \ge 1}$ or, simply, a \emph{reciprocal Puiseux monoid}. By definition, every reciprocal Puiseux monoid is weak reciprocal. Observe that none of the weak reciprocal Puiseux monoids in Examples~\ref{ex:classic antimatter PM} and~\ref{ex:Grams' monoid} is a reciprocal Puiseux monoid. Let us take a look at the prototypical example.

\begin{example} \label{ex:prime reciprocal PM}
	The monoid $\big\langle \frac 1p \mid p \in \pp \big\rangle$ is the reciprocal Puiseux monoid of the strictly increasing sequence whose underlying set is $\pp$. This additive monoids seems to be first used in \cite[Example~2.1]{AAZ90} to exhibit an example of an integral domain satisfying the ACCP that is not a BFD.
\end{example}

In addition to the sequence $(d_n)_{n \ge 1}$ defined right before Example~\ref{ex:prime reciprocal PM}, let $(c_n)_{n \ge 1}$ be a sequence of positive integers with $\gcd(c_n, d_n) = 1$ for every $n \in \nn$. Then the Puiseux monoid generated by the set $\big\{ \frac{c_n}{d_n} \mid n \in \nn \big\}$ is called the \emph{almost reciprocal} Puiseux monoid of the sequence $\big( \frac{c_n}{d_n}\big)_{n \ge 1}$ or, simply, an \emph{almost reciprocal} Puiseux monoid. It is clear that every reciprocal Puiseux monoid is almost reciprocal. It turns out that almost reciprocal Puiseux monoids are atomic, and we will prove this fact soon. Therefore the Puiseux monoid in Example~\ref{ex:classic antimatter PM} is a weak reciprocal Puiseux monoid that is not almost reciprocal. On the other hand, there are almost reciprocal Puiseux monoids that are not weak reciprocal. The following example illustrates this observation.

\begin{example} \label{ex:PM that is not weak reciprocal}
	The Puiseux monoid $M = \big\langle \frac{p+1}{p} \mid p \in \pp \big\rangle$ is almost reciprocal, and it follows from Proposition~\ref{prop:atomic PMs} that $M$ is atomic with set of atoms $\big\{ \frac{p+1}{p} \mid p \in \pp \big\}$. This, along with the fact that every generating set of $M$ must contain $\mathcal{A}(M)$, guarantees that $M$ is not (weak) reciprocal. 
\end{example}

In the following proposition we identify a class of atomic Puiseux monoids.  

\begin{prop} \label{prop:atomic PMs}
	 Let $(q_n)_{n \ge 1}$ be a sequence of positive rationals, and let $(p_n)_{\ge 1}$ be a sequence of primes such that $p_n \mid \mathsf{d}(q_n)$ but $p_n \nmid \mathsf{d}(q_k)$ for every $n \in \nn$ and $k \neq n$. Then the Puiseux monoid $M = \langle q_n \mid n \in \nn \rangle$ is atomic with $\mathcal{A}(M) = \{ q_n \mid n \in \nn \}$.
\end{prop}

\begin{proof}
	It suffices to prove that $\mathcal{A}(M) = \{ q_n \mid n \in \nn  \}$. For each $n \in \nn$, set $q_n = \frac{c_n}{d_n}$ for some $c_n, d_n \in \nn$ with $\gcd(c_n, d_n) = 1$. Now fix $k \in \nn$ and then write
	\begin{equation} \label{eq:identity to atomicity I}
		\frac{c_k}{d_k} = \alpha_1 \frac{c_1}{d_1} + \dots + \alpha_N \frac{c_N}{d_N}
	\end{equation}
	for some $N \in \nn_{\ge k}$ and $\alpha_1, \dots, \alpha_N \in \nn_0$. After setting $d := d_1 \cdots d_N$ and $d'_i := \frac{d}{d_i}$ for every $i \in \ldb 1,N \rdb$, we can multiply~\eqref{eq:identity to atomicity I} by $d$ to obtain the identity
	\begin{equation} \label{eq:identitty to atomicity II}
		d'_k c_k = \alpha_1 d'_1 c_1 + \dots + \alpha_N d'_N c_N.
	\end{equation}
	Observe that if $\alpha_k = 0$, then $p_k$ would divide the right-hand side of~\eqref{eq:identitty to atomicity II}, which contradicts that $\gcd(d'_k c_k, p_k) = 1$. Therefore $\alpha_k > 0$, and so we can infer from~\eqref{eq:identitty to atomicity II} that $\alpha_k = 1$ and $\alpha_i = 0$ for any $i \neq k$. Thus, $q_k = \frac{c_k}{d_k} \in \mathcal{A}(M)$. As a result, the equality $\mathcal{A}(M) = \{ q_n \mid n \in \nn \}$ holds.
\end{proof}

\begin{cor} \label{cor:almost reciprocal PMs are atomic}
	Every almost reciprocal Puiseux monoid is atomic.
\end{cor}

There are weak reciprocal Puiseux monoids that are atomic non-reciprocal Puiseux monoids. We identify a family of such monoids in the following proposition.

\begin{prop} \label{prop:a class of atomic weak reciprocal PMs}
	Let $(p_n)_{n \ge 1}$ be a strictly increasing sequence of prime numbers and fix some $\ell \in \nn$. The monoid $\lr{\frac{1}{p_n p_{n + \ell}} \mid n \in \nn}$ is atomic. 
\end{prop}

\begin{proof}
	For simplicity of notation consider $M = \lr{\frac{1}{p_{i}p_{i + \ell}} \mid i\in \nn}$. Take $k \in \nn$ and assume, by way of contradiction, that
	\begin{equation} \label{eq:equality I}
		\frac{1}{p_{k}p_{k + \ell}} = \sum_{i \in N_k}{a_i\frac{1}{p_{i}p_{i + \ell}}}
	\end{equation}
	for some coefficients $a_i \in \nn$ such that $a_i \neq 0$ only for finitely many $i \in \nn$. Let $P = \pspaces \prod\limits_{\{i \in N_k \mid a_i \neq 0\}}{\pspaces p_i p_{i + \ell}}$. Multiplying \eqref{eq:equality I} by $P$, we obtain the following equality
	\[
		\frac{P}{p_{k}p_{k + \ell}} = \sum_{i \in N_k}{a_i \frac{P}{p_i p_{i + \ell}}}.
	\]
	The fact that $p_i p_{i + \ell} \mid P$ for all $i \in N_k$ implies that $p_{k}p_{k + \ell} \mid P$, and so $p_k \mid P$. Because $a_k = 0$, it is clear that $a_{k - \ell} \neq 0$. Therefore, $\frac 1{p_{k - \ell} p_k}$ divides $\frac 1{p_k p_{k + \ell}}$ in $M$, and as a consequence, $\frac 1{p_{k - \ell} p_k} \le \frac 1{p_k p_{k + \ell}}$, which is a clear contradiction. Thus, $\frac 1{p_k p_{k + \ell}} \in \mathcal{A}(M)$ for all $k \in \nn$, and from this we conclude that $M$ is atomic.
\end{proof}

Following~\cite{GG18}, we say that a Puiseux monoid $M$ is \emph{bounded} if $M$ can be generated by a bounded subset of rational numbers, and we say that $M$ is \emph{strongly bounded} if~$M$ has a generating set $S$ such that the set $\mathsf{n}(S)$ is bounded. It is clear from the definitions that every weak reciprocal Puiseux monoid is strongly bounded and also that every strongly bounded Puiseux monoid is bounded. Therefore we have the following chain of implications:
\begin{equation} \label{eq:boundedness and reciprocal chain}
	\textbf{reciprocal} \ \Rightarrow \ \textbf{weak reciprocal} \ \Rightarrow \ \textbf{strongly bounded} \ \Rightarrow \ \textbf{bounded}
\end{equation}
We emphasize that none of the implications in this chain is reversible. Indeed, we have already noticed that the first one is not reversible (the Grams' monoid is a weak reciprocal Puiseux monoid that is not reciprocal), and the fact that the last implication is not reversible can be inferred from Example~\ref{ex:PM that is not weak reciprocal}, which provides an example of a bounded Puiseux monoid that is not strongly bounded. Finally, the following example shows that the second implication is not reversible.

\begin{example}
	Let $(p_n)_{n \ge 1}$ be a strictly increasing sequence of primes with $p_1 > 3$, and then consider the almost reciprocal Puiseux monoid $M = \big\langle \frac{2 + (-1)^n}{p_n} \mid n \in \nn \big\rangle$. It follows from Proposition~\ref{prop:atomic PMs} that~$M$ is atomic with $\mathcal{A}(M) = \big\{ \frac{2 + (-1)^n}{p_n} \mid n \in \nn \big\}$. Therefore $M$ is not weak reciprocal.
\end{example}

Although not every strongly bounded Puiseux monoid is weak reciprocal, it follows that every strongly bounded Puiseux monoid is isomorphic to a weak reciprocal Puiseux monoid.

\begin{prop}
	Every strongly bounded Puiseux monoid is isomorphic to a weak reciprocal Puiseux monoid.
\end{prop}

\begin{proof}
	Let $M$ be a strongly bounded Puiseux monoid, and let $Q := \{q_n \mid n \in \nn \}$ be a generating set of~$M$ consisting of positive rationals and satisfying that $\mathsf{n}(Q)$ is bounded. Then $\mathsf{n}(Q)$ is a nonempty finite set. Let $m$ be the least common multiple of~$\mathsf{n}(Q)$. It is clear that $\frac{1}{m}M$ is a Puiseux monoid and also that the function $M \to \frac{1}{m}M$ defined by $q \mapsto \frac{q}{m}$ is a monoid isomorphism. Finally, we see that $m^{-1}M$ is the weak reciprocal Puiseux monoid of the sequence $\big(\frac{q_n}{m} \big)_{n \ge 1}$.
\end{proof}

Unlike the second implication of the chain~\eqref{eq:boundedness and reciprocal chain}, the first and the last implications do not become equivalences up to isomorphism. The following two examples witness this observation.

\begin{example}
	Let $M$ be the Grams' monoid, as described in Example~\ref{ex:Grams' monoid}. We will verify that even though $M$ is a weak reciprocal Puiseux monoid, it is not isomorphic to any reciprocal Puiseux monoid. Every Puiseux monoid isomorphic to $M$ has the form $q M$ for some $q \in \qq_{>0}$ because, according to \cite[Proposition~3.2]{fG18}, the isomorphisms given by rational multiplication are the only isomorphisms between Puiseux monoids. It is clear that $\mathcal{A}(qM) = \big\{\frac{q}{2^n p_n} \mid n \in \nn \big\}$. Observe now that, for every $n \in \nn$ with $2^n > \mathsf{n}(q)$, it follows that
	\[
		\gcd\Big( \mathsf{d}\Big( \frac{q}{2^n p_n} \Big), \mathsf{d} \Big( \frac{q}{2^{n+1} p_{n+1}} \Big) \Big) = 2 \gcd( \mathsf{d}(q) p_n,  \mathsf{d}(q) p_{n+1}) \ge 2.
	\]
	Since every generating set of $qM$ must contain both $\frac{q}{2^n p_n}$ and $\frac{q}{2^{n+1} p_{n+1}}$, it follows that $qM$ is not a reciprocal Puiseux monoid. Hence $M$ is not isomorphic to any reciprocal Puiseux monoid.
\end{example}

\begin{example}
	Consider the Puiseux monoid $M = \big\langle \frac{p+1}{p} \mid p \in \pp \big\rangle$ of Example~\ref{ex:PM that is not weak reciprocal}. As in the previous example, we can see that any Puiseux monoid isomorphic to $M$ has the form $qM$ for some $q \in \qq_{> 0}$ and, therefore, $\mathcal{A}(qM) = \big\{ q \frac{p+1}{p} \mid p \in \pp \big\}$. However, if $(p_n)_{n \ge 1}$ is a sequence of primes such that $p_n \nmid \mathsf{n}(q)$ and $p_n > 2^n \mathsf{d}(q)$ for every $n \in \nn$, then
	\[
		\mathsf{n}\Big( q \frac{p_n + 1}{p_n} \Big) = \mathsf{n}\Big( \mathsf{n}(q) \frac{p_n + 1}{\mathsf{d}(q)} \Big) > \mathsf{n}(q) 2^n 
	\]
	for every $n \in \nn$, which implies that $qM$ is not strongly bounded. As the only Puiseux monoids isomorphic to $M$ have the form $qM$, we conclude that there are no strongly bounded Puiseux monoids in the isomorphism class of~$M$.
\end{example}

\begin{prop} \label{prop:valuation PMs are weak reciprocal}
	Let $G$ be a nontrivial additive subgroup of $\qq$. Then the following statements hold.
	\begin{enumerate}
		\item $G_{\ge 0}$ is isomorphic to a weak reciprocal Puiseux monoid.
		\smallskip
		
		\item $G_{\ge 0}$ is not isomorphic to any almost reciprocal Puiseux monoid.
	\end{enumerate}
\end{prop}

\begin{proof}
	(1) By \cite[Corollary~2.8]{rG84} the group $G$ is the ascending union of a sequence of cyclic subgroups, namely, $G = \bigcup_{n \ge 1} \zz q_n$, for some sequence  of nonzero rational numbers $(q_n)_{n \ge 1}$ such that $\frac{q_n}{q_{n+1}} \in \zz$. We can assume, without loss of generality, that $q_n$ is positive for every $n \in \nn$. Therefore $G_{\ge 0}$ is the Puiseux monoid generated by the sequence $(q_n)_{n \ge 1}$, that is, $G_{\ge 0} = \langle q_n \mid n \in \nn \rangle$. For every $n \in \nn$, write $q_n = c_n q_{n+1}$ for some $c_n \in \nn$, and deduce that $\mathsf{n}(q_{n+1}) \mid \mathsf{n}(q_n)$. Thus, we can pick $N \in \nn$ such that $\mathsf{n}(q_j) = \mathsf{n}(q_k)$ for all $j,k \ge N$. Since $\nn q_j \subseteq \nn q_N$ for every $j \in \ldb 1, N-1 \rdb$, we see that $G_{\ge 0} = \langle q_n \mid n \ge N \rangle$, and so we can assume that for some $m \in \nn$, the equality $\mathsf{n}(q_n) = m$ holds for every $n \in \nn$. Now the Puiseux monoid $\frac{1}{m} G_{\ge 0}$ is a weak reciprocal Puiseux monoid, which is clearly isomorphic to $G_{\ge 0}$.
	\smallskip
	
	(2) Suppose, by way of contradiction, that $G_{\ge 0}$ is isomorphic to an almost reciprocal Puiseux monoid. Then $G_{\ge 0}$ is atomic by Corollary~\ref{cor:almost reciprocal PMs are atomic}. Write $G_{\ge 0} = \langle q_n \mid n \in \nn \rangle$ as in the previous part, where $(q_n)_{n \ge 1}$ is a sequence of positive rationals such that $\nn q_{n+1} \subseteq \nn q_n$ for every $n \in \nn$. Since $G_{\ge 0}$ is atomic and $\nn q_{n+1} \subseteq \nn q_n$ for every $n \in \nn$, there exists $N \in \nn$ such that $q_n = q_N$ for every $n \ge N$. Hence $G_{\ge 0} = \langle q_N \rangle$, which is a contradiction because almost reciprocal Puiseux monoids are not finitely generated.
\end{proof}

As a consequence of Proposition~\ref{prop:valuation PMs are weak reciprocal}, we obtain that every Puiseux monoid can be embedded into a Puiseux monoid that is weak reciprocal up to isomorphism. Moreover, we can readily observe that for every Puiseux monoid $M$ we have an embedding $M \hookrightarrow \omega(M)$, where $\omega(M)$ is the weak reciprocal Puiseux monoid obtained by replacing the numerator of every element of $M^\bullet$ by~$1$. Let us record this observation.

\begin{remark} \label{rem:PMs can be embedded into weak reciprocal PMs}
	Every Puiseux monoid can be embedded into a Puiseux monoid that is weak reciprocal.
\end{remark}

In contrast to Remark~\ref{rem:PMs can be embedded into weak reciprocal PMs}, there are Puiseux monoids that cannot be embedded into any almost reciprocal Puiseux monoid. We conclude this section with some of such examples.

\begin{example}
	Take $q \in \qq_{> 0} \setminus \nn$, 
	and consider the Puiseux monoid $M_q = \langle q^n \mid n \in \nn \rangle$. Suppose, towards a contradiction, that $M_q$ is a submonoid of an almost reciprocal Puiseux monoid $M$. Since $M$ is almost reciprocal, it is atomic by Corollary~\ref{cor:almost reciprocal PMs are atomic}, and so we can write $M = \langle r_n \mid n \in \nn \rangle$, where $\mathcal{A}(M) = \{r_n \mid n \in \nn\}$. Let $p$ be a prime divisor of $\mathsf{d}(q)$. Since $M_q \subseteq M$ and $M$ is an almost reciprocal Puiseux monoid, there exists exactly one $n \in \nn$ such that $p \mid \mathsf{d}(r_n)$. Let $m$ be the maximum exponent such that $p^m$ divides $\mathsf{d}(r_n)$. Then $p^{m+1}$ does not divide $\mathsf{d}(r)$ for any $r \in M$, which contradicts that $q^{m+1} \in M$. Thus, $M_q$ cannot be embedded into any almost reciprocal Puiseux monoid.
\end{example}

\bigskip
\section{The Ascending Chain Condition on Principal Ideals}
\label{sec:ACCP}

In this section, we introduce the notion of an atomic decomposition, which is a relaxation of that of a factorization. Let $M$ be an atomic Puiseux monoid with set of atoms $\mathcal{A}(M) = \{a_n \mid n \in \nn\}$. For each $x \in M$, we say that
\[
	x = N + \sum\limits_{i \in \nn} c_i a_i
\]
is an \emph{atomic decomposition} of $x$ provided that $N \in \nn_0$ and $c_i \in \ldb 0, \mathsf{d}(a_i) - 1 \rdb$ 
for every $i \in \nn$, where only finitely many coefficients $c_i$ are nonzero. Observe that if $1 \in \mathcal{A}(M)$, then the factorizations of $x$ are in natural bijection with the atomic decompositions of~$x$. 

\begin{definition}
	We say that an atomic Puiseux monoid $M$ has \emph{unique atomic decomposition} provided that every element of $M$ has a unique atomic decomposition.
\end{definition}

As the following example illustrates, there are atomic Puiseux monoids that do not have unique atomic decompositions.

\begin{example}
	Set $q := 2/3$, and consider the Puiseux monoid $M := \langle q^n \mid n \in \nn_0 \rangle$. It is well-known and not too difficult to argue that $M$ is an atomic monoid with $\mathcal{A}(M) = \{q^n \mid n \in \nn_0\}$ (see \cite[Proposition~4.3]{CGG21}). Since $2 = q + 3q^2$ in $M$, it follows that $M$ does not have unique atomic decomposition.
\end{example}

 It turns out that every almost reciprocal Puiseux monoid has unique atomic decomposition.

\begin{prop} \label{prop:almost reciprocal PMs have the UAD}
	Every almost reciprocal Puiseux monoid has unique atomic decomposition.
\end{prop}

\begin{proof}
	Let $(c_n)_{n \ge 1}$ and $(d_n)_{n \ge 1}$ be two sequences of positive integers satisfying that $\gcd(c_n, d_n) = 1$ for every $n \in \nn$ and $\gcd(d_m, d_n) = 1$ for all $m,n \in \nn$ with $m \neq n$. Now consider the almost reciprocal Puiseux monoid
	\[
		M := \Big\langle \frac{c_n}{d_n} \ \Big{|} \ n \in \nn \Big\rangle.
	\]
	To argue that the Puiseux monoid $M$ has unique atomic decomposition, take $q \in M^\bullet$ and write
	\[
		q = N + \sum_{i=1}^n \alpha_i \frac{c_i}{d_i} \quad \text{ and } q = N' + \sum_{i=1}^n \alpha'_i \frac{c_i}{d_i}
	\]
	for some $N, N' \in \nn_0$ and nonnegative integer coefficients $\alpha_1, \dots, \alpha_n$ and $\alpha'_1, \dots, \alpha'_n$ with $\alpha_i, \alpha'_i \in \ldb 0, d_i - 1 \rdb$ for every $i \in \ldb 1,n \rdb$ (for some $n \in \nn$ large enough). After setting $D = d_1 \cdots d_n$ and $D_i := D/d_i$ for each $i \in \ldb 1,n \rdb$, we see that
	\begin{equation} \label{eq:UAD aux equation}
		(N - N')D = \sum_{i=1}^n (\alpha'_i - \alpha_i)D_i c_i. 
	\end{equation}
	It follows from~\eqref{eq:UAD aux equation} that, for each $j \in \ldb 1,n \rdb$, the expression $(\alpha'_j - \alpha_j)D_j c_j$ is divisible by $d_j$, and so the fact that $\gcd(d_j, D_j c_j) = 1$ ensures that $d_j \mid \alpha'_j - \alpha_j$. As a result, $\alpha'_j = \alpha_j$ for every $j \in \ldb 1,n \rdb$, and so $N' = N$. Hence every element of $M$ has a unique atomic decomposition, and the proposition follows.
\end{proof}

Unlike almost reciprocal Puiseux monoids, weak reciprocal Puiseux monoids may not have unique atomic decomposition even when they are atomic. The following example illustrates this observation.

\begin{example}
	Let $(p_n)_{n \ge 1}$ be a strictly increasing sequence of primes, and consider the Puiseux monoid
	\[
		M = \Big\langle \frac{1}{p_n p_{n+1}} \ \Big{|} \ n \in \nn \Big\rangle.
	\]
	It follows from Proposition~\ref{prop:a class of atomic weak reciprocal PMs} that the weak reciprocal Puiseux monoid $M$ is atomic with $\mathcal{A}(M) = \big\{  \frac{1}{p_n p_{n+1}} \mid n \in \nn \big\}$. Since
	\[
		\frac{1}{p_n} = p_{n-1} \frac{1}{p_{n-1} p_n} = p_{n+1} \frac{1}{p_n p_{n+1}}
	\]
	for every $n \in \nn$, we see that $M$ is an atomic weak reciprocal Puiseux monoid without unique atomic decomposition.
\end{example}

If an atomic Puiseux monoid $M$ with $\mathcal{A}(M) = \{a_n \mid n \in \nn\}$ has unique atomic decomposition, then we can define functions $\eta, \zeta_n \colon M \to \nn_0$ for each $n \in \nn$ such that, for each $q \in M$,
\begin{equation} \label{eq:atomic decomposition}
	q = \eta(q) + \sum\limits_{i \in \nn} \zeta_n(q) a_n
\end{equation}
is the atomic decomposition of $q$. It follows from the definition of the functions $\zeta_n$'s that, for each $q \in M$, the sequence $(\zeta_n(q))_{n \ge 1}$ has only finitely many nonzero terms. It turns out that every Puiseux monoid having unique atomic decomposition satisfies the ACCP.

\begin{theorem}\label{theorem:UniqueDecompositionACCP}
	Every almost reciprocal Puiseux monoid satisfies the ACCP.
\end{theorem}

\begin{proof}
	Since Puiseux monoids are reduced, it immediately follows that if a Puiseux monoid satisfies the ACCP, then each of its submonoids also satisfies the ACCP (cf.~\cite[Proposition~2.1]{aG74}). Therefore it suffices to argue the statement of the theorem for reciprocal Puiseux monoids; although this does not make our argument any easier, it will definitely simplify the notation. 
	
	Let $(d_n)_{n \ge 1}$ be a sequence of positive integers satisfying that $\gcd(d_m, d_n) = 1$ for all $m,n \in \nn$ with $m \neq n$, and consider the reciprocal Puiseux monoid
	\[
		M = \Big\langle \frac{1}{d_n} \ \Big{|} \ n \in \nn \Big\rangle.
	\]
	Assume, by way of contradiction, that there exists an ascending chain $(q_n + M)_{n \ge 1}$ of principal ideals of~$M$ that does not stabilize. After replacing $(q_n)_{n \in \nn}$ by one of its subsequences, one can further assume that $(q_n + M) \subsetneq (q_{n + 1} + M)$ for every $n \in \nn$. Since~$M$ is an almost prime reciprocal Puiseux monoid, it follows from Proposition~\ref{prop:almost reciprocal PMs have the UAD} that $M$ has the unique atomic decomposition. Then we can consider the functions $\eta, \zeta_i \colon M \to \nn_0$ defined via the identities~\eqref{eq:atomic decomposition}. Now set $r_{n+1} := q_n - q_{n + 1}$ for every $n \in \nn$. Clearly, $\eta(q_n) = \eta(q_{n+1} + r_{n+1}) \ge \eta(q_{n+1})$. Therefore there is an $m \in \nn$ such that $\eta(q_n) = \eta(q_{n + 1})$ for every $n \geq m$. This, along with the fact that $q_{n+1}$ strictly divides $q_n$ in $M$, guarantees that $\sum_{i \in \nn} \zeta_i(q_n) > \sum_{i \in \nn} \zeta_i(q_{n+1})$ for every $n \geq m$. So there exists $\ell \in \nn$ such that $\sum_{i \in \nn} \zeta_i(q_n) = 0$ for every $n \geq \ell$, which implies that the sequence $(q_n)_{n \in \nn}$ is constant from one point on. However, this contradicts that the chain of ideals $(q_n + M)_{n \ge 1}$ does not stabilize. Thus, $M$ satisfies the ACCP.
\end{proof}

Although every reciprocal Puiseux monoid satisfies the ACCP, there are atomic weak reciprocal Puiseux monoids that do not satisfy the ACCP.

\begin{remark} \label{rem:Grams monoid is not ACCP}
	It is well known that the Grams' monoid $M$ does not satisfy the ACCP as the chain of principal ideals $\big( \frac{1}{2^n} + M \big)_{n \ge 1}$ does not stabilize.
\end{remark}

In addition, the atomic weak reciprocal Puiseux monoids introduced in Proposition~\ref{prop:a class of atomic weak reciprocal PMs} do not satisfy the ACCP, as we proceed to argue.

\begin{prop} \label{prop:a class of atomic weak reciprocal PMs without the ACCP}
	Let $(p_n)_{n_ \in \nn}$ be a strictly increasing sequence of prime numbers and fix some $\ell \in \nn$. Then $\big\langle \frac{1}{p_{i}p_{i + \ell}} \mid i\in \nn \big\rangle$ is an atomic weak reciprocal Puiseux monoid that does not satisfy the ACCP.
\end{prop}

\begin{proof}
	Set $M := \big\langle \frac{1}{p_n p_{n + \ell}} \mid n \in \nn \big\rangle$. We have already proved in Proposition~\ref{prop:a class of atomic weak reciprocal PMs} that $M$ is atomic. Let us argue that $M$ does not satisfy the ACCP. To do so, consider the chain of principal ideals $\big( \frac 1{p_{\ell n}} + M \big)_{n \ge 1}$. For each $n \in \nn$, we see that $p_{\ell(n+1)} - p_{\ell n} \in \nn$ and, therefore, the equality
	\[
		\frac{1}{p_{\ell n}} - \frac{p_{\ell(n+1)} - p_{\ell n}}{p_{\ell n}p_{\ell(n+1)}} = \frac{1}{p_{\ell(n+1)}}
	\]
	guarantees that $\frac 1{p_{\ell(n+1)}}$ strictly divides $\frac 1{p_{\ell n}}$ in $M$. As a consequence, $\big( \frac 1{p_{\ell n}} + M \big)_{n \ge 1}$ is an ascending chain of principal ideals of $M$ that does not stabilize. Hence $M$ does not satisfy the ACCP.
\end{proof}

\smallskip
\subsection{Sets of Lengths}

Recall that a Puiseux monoid $M$ is a BFM if it is atomic and $\mathsf{L}_M(q)$ is finite for every $q \in M$. Weak reciprocal Puiseux monoids are not BFMs

\begin{prop} \label{prop:Sets of Lengths}
	Let $M$ be an atomic Puiseux monoid.
	\begin{enumerate}
		\item If $M$ is weak reciprocal, then $|\mathsf{L}_M(1)| = \infty$, and so $M$ is not a BFM.
		\smallskip
		
		\item If $M$ is reciprocal and $q \in M$, then $|\mathsf{L}_M(q)| \in \{1,\infty\}$. More specifically, the following statements hold.
		\smallskip
		\begin{itemize}
			\item $|\mathsf{L}_M(q)| = \infty$ if $1 \mid_M q$.
			\smallskip
			
			\item $|\mathsf{L}_M(q)| = 1$ if $ \ 1 \nmid_M q$.
		\end{itemize}
	\end{enumerate}
\end{prop}

\begin{proof}
	(1) Suppose that $M$ is weak reciprocal. Because $M$ is atomic, there exists a strictly increasing sequence $(d_n)_{n \ge 1}$ such that $\mathcal{A}(M) = \big\{ \frac 1{d_n} \mid n \in \nn \big\}$. As $1 = d_n \frac{1}{d_n}$ for every $n \in \nn$, we see that $\{d_n \mid n \in \nn\} \subseteq \mathsf{L}_M(1)$, and so $|\mathsf{L}_M(1)| = \infty$. The second statement is an immediate consequence of the first one.
	\smallskip
	
	(2) Assume now that $M$ is a reciprocal Puiseux monoid, and fix $q \in M$. Write $M = \big\langle \frac{1}{d_n} \mid n \in \nn \rangle$ for some strictly increasing sequence $(d_n)_{n \ge 1}$ of positive integers such that $\gcd(d_m, d_n) = 1$ when $m \neq n$. It follows from Proposition~\ref{prop:almost reciprocal PMs have the UAD} that $M$ has unique atomic decomposition.
	\smallskip
	\begin{itemize}
		\item Assume first that $1 \mid_M q$. Then the fact that $|\mathsf{L}_M(1)| = \infty$ (by part~(1)) implies that $|\mathsf{L}_M(q)| = \infty$.
		\smallskip
		
		\item Now assume that $1 \nmid_M q$. Then $\eta(q) = 0$ in the unique atomic decomposition of $q$. As a result, there exist $\alpha_1, \dots, \alpha_n \in \nn_0$ such that $q = \sum_{k=1}^n \alpha_k \frac{1}{d_k}$, where $\alpha_k \in \ldb 0, d_k - 1 \rdb$ for every $k \in \ldb 1,n \rdb$. Since every factorization $z \in \mathsf{Z}(M)$ using at least $d_i$ copies of the atom $\frac{1}{d_i}$ for some $i \in \nn$ yields an atomic decomposition of $\pi(z)$ with $\eta(\pi(z)) \ge 1$ (here $\pi$ is the factorization homomorphism of $M$), every factorization of $q$ must be an atomic decomposition. As a result, $|\mathsf{Z}_M(q)| = 1$, which implies that $|\mathsf{L}_M(q)| = 1$.
	\end{itemize}
\end{proof}

%

As the following example indicates, the second statement of Proposition~\ref{prop:Sets of Lengths} does not hold if we replace the reciprocal condition by the weak reciprocal condition.

\begin{example}
	Let $(p_n)_{n \in \nn}$ be a strictly increasing sequence of odd primes and take $k \in \nn$. Consider the monoid
	\[
		M = \Big\langle \frac{1}{2^m p_m}, \frac{1}{p_n} \ \Big{|} \ m \in \llbracket 1, k \rrbracket\ \text{ and }\ n \in \nn_{>k} \Big\rangle.
	\]
	It follows from Proposition~\ref{prop:a class of atomic weak reciprocal PMs} that $M$ is atomic with
	\[
		\mathcal{A}(M) = \Big\{  \frac{1}{2^m p_m}, \frac{1}{p_n} \ \Big{|} \ m \in \llbracket 1, k \rrbracket\ \text{ and }\ n \in \nn_{>k} \Big\}.
	\]
	Since $\frac{1}{2} = 2^{m - 1}p_m \frac{1}{2^m p_m}$ for every $m \in \ldb 1,k \rdb$, we conclude that $|\mtl\big( \frac 12 \big) | \ge k$. On the other hand, consider an arbitrary factorization of $\frac 12$, namely,
	\begin{equation} \label{eq:atomic decomposition I}
		\frac{1}{2} = \sum_{i = 1}^{k} c_i \frac{1}{2^i p_i} + \sum_{i = k + 1}^N c_i \frac{1}{p_i}
	\end{equation}
	for some $k,N \in \nn$ with $N > k$ and $c_1, \dots, c_N \in \nn_0$. For each $i \in \ldb k+1, N \rdb$, the fact that $\frac{c_i}{p_i} \le \frac 12$ ensures that $c_i < p_i$. Then after multiplying both sides of~\eqref{eq:atomic decomposition I} by $P := 2^k \prod_{i=1}^N p_i$ and carrying out obvious algebraic manipulations, we find that $c_i = 0$ for every $i \in \ldb k+1, N \rdb$. This along with the fact that $c_i \le 2^{k-1}p_k$ for every $i \in \ldb 1, k \rdb$ allows us to conclude that $\frac{1}{2}$ has finitely many factorizations in $M$ and, therefore, that $\mathsf{L}\big( \frac 12 \big)$ is a finite set. Hence $|\mathsf{L}\big(\frac 12 \big)| \notin \{1, \infty\}$.
%
%
\end{example}

Let $M$ be an atomic monoid. For $x \in M$ and $\ell \in \nn$, we set
\[
	\mathsf{Z}(x, \ell) := \{z \in \mathsf{Z}(x) \mid |z| = \ell \}.
\]
Following Geroldinger and Zhong~\cite{GZ21}, we say that a Puiseux monoid $M$ is a \emph{length-finite-factorization monoid} (or an \emph{LFFM}) if $M$ is atomic and $\mathsf{Z}(x, \ell)$ is finite for all $x \in M$ and $\ell \in \nn$. We observe that a monoid is an FFM if and only if it is both a BFM and an LFFM. Thus, the length-finite-factorization condition is what a BFM needs to satisfy in order to be an FFM. It follows from Proposition~\ref{prop:Sets of Lengths} that no weak reciprocal Puiseux monoid is a BFM. By contrast, in~\cite{GZ21} the authors show that the reciprocal Puiseux monoid $\langle \frac{1}{p} \mid p \in \pp \rangle$ is an LFFM. We know that every BFM satisfies the ACCP. It is natural to wonder whether every LFFM satisfies the ACCP. However, the only example known of a Puiseux monoid satisfying the length-finite-factorization condition is $\langle \frac{1}{p} \mid p \in \pp \rangle$, which also satisfies the ACCP. In order to produce examples of LFFMs without the ACCP, we extend \cite[Example~2.4]{GZ21} to the class of almost reciprocal Puiseux monoids, our proof following the idea given in \cite[Example~2.4]{GZ21} by Geroldinger and Zhong.

\begin{prop} 
	Let $(q_n)_{n \ge 1}$ be a sequence of positive rationals, and let $(p_n)_{\ge 1}$ be a sequence of primes such that $p_n \mid \mathsf{d}(q_n)$ but $p_n \nmid \mathsf{d}(q_k)$ for every $n \in \nn$ and $k \neq n$. Then the Puiseux monoid $M = \langle q_n \mid n \in \nn \rangle$ is an LFFM. 
\end{prop}

\begin{proof}
	We have proved in Proposition~\ref{prop:atomic PMs} that $M = \langle q_n \mid n \in \nn \rangle$ is atomic with $\mathcal{A}(M) = \{ q_n \mid n \in \nn \}$. To argue that $M$ is an LFFM, fix $q \in M^\bullet$ and $\ell \in \nn$. Take $k \in \nn$ large enough so that $p_i \nmid \mathsf{d}(q)$ and $p_i > \ell$ for every $i \ge k$. Now suppose that $q$ has an $\ell$-length factorization in $M$, namely,
	\[
		q = \sum_{i=1}^n c_i q_i,
	\]
	for some $c_1, \dots, c_n \in \nn_0$ with $\sum_{i=1}^n c_i = \ell$. If $n > k$ and $i \in \ldb k+1, n \rdb$, then the fact that $p_i \nmid \mathsf{d}(q)$ guarantees that $p_i \mid c_i$, and so the fact that $p_i > \ell \ge c_i$ ensures that $c_i = 0$. Hence we can assume that $n \le k$, which immediately implies that there are only finitely many factorizations of $q$ with length~$\ell$. Hence $M$ is an LFFM.
\end{proof}

\begin{cor} \label{cor:almost reciprocal PM are LFFM}
	Every almost reciprocal Puiseux monoid is an LFFM.
\end{cor}

As a special case of Corollary~\ref{cor:almost reciprocal PM are LFFM}, we see that the Grams' monoid is a Puiseux monoid satisfying the length-finite-factorization condition but not the ACCP. We finally observe that the statement resulting from replacing the almost reciprocal condition by the weak reciprocal condition in Corollary~\ref{cor:almost reciprocal PM are LFFM} is not true as we have already seen examples of weak reciprocal Puiseux monoids that are not even atomic.

\bigskip
\section*{Acknowledgments}

The authors would like to thank Felix Gotti for the extraordinary encouragement and guidance during the preparation of this paper. The third author gratefully acknowledges to be part of the MIT UROP (under the mentorship of Felix Gotti) during Fall 2020, when he started collaborating on the research project that later evolved to this paper.

\bigskip

\end{document}